\documentclass{siamltex}
\usepackage{amsmath}
\usepackage{amssymb}
\usepackage{graphicx,color}
\usepackage{units}

\def\storage{c_s}
\def\biotwillis{\alpha}
\newcommand{\Hdiv}{H^{\text{div}}}
\newcommand{\Vdisp}{\mathbf V}
\newcommand{\disp}{\mathbf u}
\newcommand{\disptest}{\mathbf v}
\newcommand{\bcdispD}{\disp\mathrm{D}}
\newcommand{\bcdispN}{\disp\mathrm{N}}
\newcommand{\Vdarcy}{\mathbf W}
\newcommand{\darcy}{\mathbf w}
\newcommand{\darcytest}{\mathbf z}
\newcommand {\DD}{\varepsilon}
\renewcommand{\div}{\nabla\!\cdot\!}

\newcommand{\bfA}{{\bf A}}

\newcommand{\bfE}{{\bf E}}

\newcommand{\bfI}{{\bf I}}

\newcommand{\bfeta}{{\mbox{\boldmath $\eta$}}}

\newcommand{\bfsigma}{{\mbox{\boldmath $\sigma$}}}

\newcommand{\bftau}{{\mbox{\boldmath $\tau$}}}

\newcommand{\bfphi}{{\mbox{\boldmath $\phi$}}}

\newcommand{\bfchi}{{\mbox{\boldmath $\chi$}}}

\newcommand{\bff}{{\bf f}}

\newcommand{\bfn}{{\bf n}}

\newcommand{\bfq}{{\bf q}}

\newcommand{\bfx}{{\bf x}}

\newcommand{\Th}{\mathbb T}

\newenvironment{theorem*}{{\bf Theorem}\em}{\rm\mbox{}}
\newtheorem{assumption}[theorem]{Assumption}
\newtheorem{prop}[theorem]{Proposition}
\newtheorem{remark}[theorem]{Remark}

\begin{document}

\bibliographystyle{siam}

\title{A finite element method with strong mass conservation for Biot's linear consolidation model}

\author{Guido Kanschat\thanks{Interdisciplinary Center for Scientific Computing (IWR), Heidelberg University, Heideberg, Germany.}\and
  Beatrice Riviere\thanks{Department of Computational and Applied
    Mathematics, Rice University, Houston, TX 77005. }}

\date{\today}
\maketitle

\begin{abstract}
An H(div) conforming finite element method for solving the linear Biot equations is analyzed. 
Formulations for the standard mixed method are combined with
formulation of interior penalty discontinuous Galerkin method to obtain a consistent scheme.
Optimal convergence rates are obtained.
\end{abstract}
%

\section{Introduction}

In this article, we present a new finite element discretization of a
linear model for poroelasticity~\cite{Biot41}. The main features of
our approach are a consistent coupling of fluid and solid velocity
without projection and consistent approximation rates for both
velocity fields and the fluid pressure. Thus, our scheme is robust
with respect to fluid and solid compressibility manifested by the
storage coefficient $\storage$ and the Biot-Willis constant
$\biotwillis$, and we obtain optimal convergence rates in $L^2$ for
pressure and velocity. We achieve this by using a standard mixed
formulation based on $\Hdiv$-conforming finite element spaces with
matching pressure for the fluid velocity and by using the same vector
space combined with discontinuous Galerkin flux terms for
$H^1$-consistency.

Already in 1994, Murad and Loula~\cite{muradloula} analyze the case
with $\storage = 0$ (incompressible fluid). They use $H^1$-conforming
finite elements for displacement and fluid pressure, and obtain
estimates of Taylor-Hood type, that is, for pressure shape functions
of degree $k-1$ and displacement of degree $k$, they have balanced
approximation in $L^2$ for strain and pressure of order
$h^k$. Assuming additional regularity, duality yields that the
displacement converges of order $h^{k+1}$, while by taking
derivatives, the seepage velocity is of order $h^{k-1}$. It is this
gap in approximation, we are overcoming with our method.

In~\cite{OyarzuaRuizBaier16}, Oyarzua and Ruiz-Baier introduce
a ``total'' pressure $\phi = p-\lambda \div u$ in order to treat the
coupling between solid and fluid in a more robust way. Since they
compute the pressure $p$ as well, this amounts to adding a variable
for the dilation $\div u$. They obtain for a Taylor-Hood approximation
of degree $k/k-1$ of the displacement/total pressure pair and a
pressure approximation of degree $k$ an energy estimate of order $k$
involving $H^1$-norms of the displacement and pressure and the
$L^2$-norm of the total pressure. Thus, assuming elliptic regularity,
the $L^2$-error of the displacement is only one order better than that
of the fluid velocity.
A similar gap can be observed in \cite{yi2013coupling}, where the
error of the displacement gradient is in balance with the seepage
velocity. The discretization there is more similar to ours though,
since it uses Raviart-Thomas elements for the seepage
velocity. Different from here, a nonconforming element is used there
for the solid displacement. Estimates of the same kind were obtained
in~\cite{phillips2007coupling,phillips2008coupling} for continuous and
discontinuous Galerkin approximation of the solid displacement,
respectively, but under the restrictive assumption $\storage>0$, which
excludes incompressible fluids.


 
In~\cite{yi2014convergence}, a mixed method involving discretization
of pressure, seepage velocity, ``total'' stress, and displacement is
used. The finite elements are Raviart-Thomas pairs for velocity and
pressure and Arnold-Winther pairs for stress and displacement. It is
to our knowledge the only other result which produces equal order
approximation in $L^2$ for velocity and displacement, if matching
polynomial degrees are chosen. Compared to the method proposed here,
introducing a discretization of the total stress increases the number
of degrees of freedom considerably. In addition, the optimal error
estimate in $L^\infty(L^2)$ there is only obtained for $\storage>0$,
while it deteriorates to $L^2(L^2)$ for $\storage=0$, while our
analysis does not suffer from this problem and holds in every
timestep. Finally, robustness with respect to all involved parameters
of discretizations based on Raviart-Thomas pairs is discussed
in~\cite{HongKraus17}, and their analysis of the system to be solved
in a single time step applies to our method as well. Not using our
assumption~\ref{ass:hdiv} below, they choose the order of the displacement
space higher than that of the seepage velocity space.




The remainder of this article is organized as follows: in
Section~\ref{sec:model-problem}, we denote Biot's consolidation
equations in displacement, pressure, and seepage velocity
variables. Then, in section~\ref{sec:cont-time-scheme}, we state a
semidiscrete scheme and present its error analysis in
Section~\ref{sec:priori-error-estim}. A simple time discretization and
its analysis are provided in Section~\ref{sec:discrete-time-scheme},
and we conclude with numerical tests in Section~\ref{sec:numer-exper}.

\section{Model problem}
\label{sec:model-problem}

The linear Biot system coupling the deformation $\disp$ of the porous media, the fluid pressure $p$, and the discharge or seepage velocity $\darcy$ of the fluid is written as:
\begin{xalignat}2
  \frac{\partial}{\partial t} (\storage p + \biotwillis  \div \disp) +\div \darcy &= f_1,
  &\text{in }& \Omega\times (0,T),
  \label{eq:pdemass}\\
  K^{-1} \darcy &= -\nabla p,&\text{in }& \Omega\times (0,T),
\label{eq:pdedarcy}\\
-\div (\bfsigma -\biotwillis  p \bfI) &= \bff_2,& \text{in }& \Omega\times (0,T).
\label{eq:momentum}
\end{xalignat}
The constant $\biotwillis $ is called the Biot-Willis
constant~\cite{BiotWillis1957}, which represents unaccounted volume
changes due to a third phase, for instant small air inclusions in
soil.  It takes a value very close to one.  The constant $\storage$
represents the constrained specific storage coefficient (see
\cite{showalter2010poroelastic} and references therein) and is related
to compressibility of the fluid. Therefore, it is close to zero in
many applications.  The permeability $K$ is a symmetric positive
definite matrix.  We assume here that the effective stress tensor
satisfies Hooke's law:
\begin{gather*}
  \bfsigma = \lambda (\div \disp) \bfI + 2\mu \DD(\disp),  
\end{gather*}
where
\begin{gather*}
  \DD(\disp) = \frac12 (\nabla \disp + (\nabla \disp)^T).  
\end{gather*}

The system is completed by initial conditions
\begin{gather}
  \label{eq:initial}
  p(0) = p_0,\quad\disp(0) = \disp_0 \quad \text{in} \quad \Omega.  
\end{gather}
such that equation~\eqref{eq:momentum} is satisfied at $t=0$. No
initial condition on $\darcy$ is required since it is only coupled
algebraically. In practice, the initial pressure is experimentally
measured and the displacement $\disp_0$ is obtained by
solving~\eqref{eq:momentum}.

The boundary of the domain is decomposed into two pairs of disjoint sets:
\[
\partial\Omega = \Gamma_{p\mathrm{D}} \cup\Gamma_{p\mathrm{N}}
 = \Gamma_{\bcdispD} \cup\Gamma_{\bcdispN},
\]
with
\[
\Gamma_{p\mathrm{D}} \cap\Gamma_{p\mathrm{N}}
 = \Gamma_{\bcdispD} \cap\Gamma_{\bcdispN} = \emptyset.
\]
We prescribe the pressure and velocity on the boundary
\begin{eqnarray}
p = p_\mathrm{D}, \quad \mbox{on} \quad\Gamma_{p\mathrm{D}},\\
\darcy \cdot\bfn = 0,\quad \mbox{on}\quad\Gamma_{p\mathrm{N}},
\end{eqnarray}
and we prescribe the displacement and total normal stress
\begin{eqnarray}
\disp = \disp_{\mathrm{D}}, \quad \mbox{on} \quad\Gamma_{\bcdispD},\\
(\bfsigma -\biotwillis  p \bfI)\bfn = \bfsigma_{\mathrm{N}}, \quad \mbox{on}\quad\Gamma_{\bcdispN}.
\label{eq:dispbcN}
\end{eqnarray}
Throughout the paper, the unit normal (resp. tangential) vector to the boundary $\partial\Omega$ is denoted
by $\bfn$ (resp. $\bftau$).
We remark that the boundary condition $\darcy\cdot\bfn = 0$ can be
changed to the inhomogeneous boundary condition $\darcy\cdot\bfn = g$.
In that case, the datum $g$ needs to be lifted following a standard
technical argument. Furthermore, the deformation may admit more complex
boundary conditions, see for instance the numerical experiments. We
make the following assumptions:
\begin{enumerate}
\item Neither $\Gamma_{pN} = \partial\Omega$, nor is $\disp\cdot\bfn$
  prescribed on the whole boundary. This is a technical assumption
  which guarantees that neither $\div \darcy$, nor $\div \disp$ are
  forced to have mean value zero.
\item The boundary condition on $\disp$ itself is sufficient to admit
  Korn's inequality
  \begin{gather*}
    \Vert \nabla \disp \Vert_\Omega \le C \Vert \epsilon(\disp) \Vert_\Omega,
  \end{gather*}
where $\Vert \cdot\Vert_\Omega$ denotes the $L^2$ norm over $\Omega$.
  In particular, the boundary conditions must exclude solid translations and
  rotations of the whole domain.
\end{enumerate}

\section{Continuous-in-time Scheme}
\label{sec:cont-time-scheme}

Let $\Th_h$ be a shape regular family of
conforming subdivisions of $\Omega$ into simplices, parallelograms or
parallelepipeds.  Denote by $h_T$ the diameter of an element $T$ and
denote by $h$ the maximum diameter over all mesh elements. Denote by
$\Gamma_i$ the set of faces that are interior to $\Omega$.  For all
$t\geq 0$, we seek a solution $(p_h,\darcy_h,\disp_h)$ in
$Q_h\times\Vdarcy_h\times\Vdisp_h$.  The pair $(\Vdarcy_h,Q_h)$ is the
usual pair of a divergence-conforming velocity space
$\Vdarcy_h\subset \Hdiv_{0,\Gamma_{p\mathrm{N}}}(\Omega)$ and its
corresponding pressure space $Q_h\subset L^2(\Omega)$.  We denote
\[
\Vdarcy = \Hdiv_{0,\Gamma_{p\mathrm{N}}}(\Omega)
= \{ \darcytest \in \Hdiv(\Omega): \, \darcytest\cdot\bfn = 0
\mbox{ on } \Gamma_{p\mathrm{N}}\},
\]
and
\[
\Vdisp = \{\disptest \in \Hdiv(\Omega): \, \disptest\cdot\bfn = {\bf 0} \mbox{ on } \Gamma_{\bcdispD}\}. 
\]
We use the notation $\Vert \cdot\Vert_{\mathcal{O}}$ for
the $L^2$ norm on any domain $\mathcal{O}$. The $L^2$ inner-product on $\mathcal{O}$ is denoted by
$(\cdot,\cdot)_{\mathcal{O}}$.
The space $\Vdisp_h$ is a finite-dimensional subspace of $\Vdisp \cap H^1(\Th_h)$, where
$H^1(\Th_h)$ is the broken Sobolev space.
We denote by $k$ the polynomial degree for the space $Q_h$.  The space
$\Vdarcy_h$ only differs from $\Vdisp_h$ by the location of the boundary conditions, and
thus has the same order as $\Vdisp_h$.
We also assume that the spaces $\Vdarcy_h$ and $Q_h$ satisfy:
\begin{equation}
\div\Vdarcy_h = Q_h.
\label{eq:pressurespace}
\end{equation}
Therefore we also have 
\begin{equation}
\div\Vdisp_h = Q_h.
\label{eq:Vhpressurecond}
\end{equation}

Next we introduce an approximation operator $\pi_h$ satisfying for all $\darcytest\in \Vdarcy +\Vdisp$
\begin{eqnarray}
(\div \pi_h(\darcytest),q) = (\div\darcytest,q),\quad \forall q\in Q_h,\label{eq:prop1}\\
\Vert \pi_h(\darcytest)-\darcytest\Vert_{H^r(T)} \leq C h_T^{k+1-r} \vert \darcytest \vert_{H^{k+1}(T)},
\quad \forall T\in\Th_h,\quad 0\leq r\leq k,
\\
\Vert \div(\pi_h(\darcytest)-\darcytest))\Vert_{L^2(T)} \leq C h_T^{k+1} \vert \div\darcytest \vert_{H^{k+1}(T)},
\quad \forall T\in\Th_h.
\end{eqnarray}
Since the spaces $\Vdarcy$ and $\Vdisp$ differ because of the location of the boundary conditions, we also require that
\begin{eqnarray*}
\forall \darcytest\in\Vdarcy,\quad\pi_h(\darcytest)\in\Vdarcy_h,
\\
\forall \darcytest\in\Vdisp,\quad\pi_h(\darcytest)\in\Vdisp_h.
\end{eqnarray*}
We now introduce jump $[\cdot]$ and average $\{\cdot\}$ of a scalar function 
$\phi$ across an interior face $F$.  We first associate with each face $F$ in $\Gamma_i$
a unit normal vector $\bfn_F$, and we denote by $T_-$ and $T_+$ the elements that share $F$, such that $\bfn_F$ points from $T_-$ to $T_+$. We then define
\[
[\phi] = \phi|_{T_-} - \phi|_{T_+}, \quad
\{\phi\} = \frac12 (\phi_{T_-}+\phi_{T_+}).
\]
Jump and average of vector function $\bfphi$ are defined component-wise.
The $L^2$ inner-product on an open domain $\mathcal{O}$ is denoted by
$(\cdot,\cdot)_{\mathcal{O}}$. We will also use the following notation
for the inner-products on elements and faces:
\[
(\phi,\psi)_{\Th_h} = \sum_{T\in\Th_h} (\phi,\psi)_T,\quad
(\phi,\psi)_{\Gamma_i} = \sum_{F\in\Gamma_i} (\phi,\psi)_F,\quad
\]
\[
(\phi,\psi)_{\Gamma_{\bcdispD}} = \sum_{F\in\Gamma_{\bcdispD}} (\phi,\psi)_F,\quad
(\phi,\psi)_{\Gamma_{\bcdispN}} = \sum_{F\in\Gamma_{\bcdispN}} (\phi,\psi)F.
\]
The discretization of the operator $-2 \div \epsilon(\disp)$ in the
nonconforming space $\Vdisp$ follows the interior penalty (SIPG)
method \cite{Arnold,Riviere2008} with the mesh dependent form:
\begin{multline*}
d_h(\disp,\disptest) = 2 (\DD(\disp),\DD(\disptest))_{\Th_h}
+\frac{\gamma}{h} ([\disp],[\disptest])_{\Gamma_i}
\\
-2 (\{ \DD(\disp)\bfn_F\}, [\disptest])_{\Gamma_i}
-2  (\{ \DD(\disptest)\bfn_F\}, [\disp])_{\Gamma_i}
\\
+\frac{\gamma}{h} (\disp,\disptest)_{\Gamma_{\bcdispD}}
-2 (\DD(\disp)\bfn, \disptest)_{\Gamma_{\bcdispD}}
-2  (\DD(\disptest)\bfn,\disp)_{\Gamma_{\bcdispD}}
,\quad
\forall \disp,\disptest\in\Vdisp.
\end{multline*}
The parameter $\gamma>0$ is the penalty parameter, chosen large enough to ensure coercivity of the bilinear form $d_h(\cdot,\cdot)$.
Since the space $\Vdisp$ is $\Hdiv$-conforming, no penalty formulation
for the term $\nabla \div \disp$ is needed. Accordingly, we define the bilinear form
\[
  a_h(\disp,\disptest) = \mu d_h(\disp,\disptest)
  + \lambda (\nabla \cdot \disp, \nabla \cdot\disptest)_\Omega,\quad
  \forall \disp,\disptest\in\Vdisp.
\]
From the equalities of these spaces, we immediately deduce the inf-sup conditions
in~\cite{BoffiBrezziFortin13} and~\cite{HansboLarson02,SchoetzauSchwabToselli03}:
\begin{alignat}6
  \forall& q\in Q_h 
  &\;\,\exists &\darcytest\in \Vdarcy_h
  & \;:\quad&
  &\div\darcytest&=q
  &\quad\wedge\quad&
  &\Vert \darcytest \Vert_{\Hdiv(\Omega)} &\le \frac1{\beta_{\Vdarcy}}\Vert q \Vert_\Omega
  \label{eq:inf-sup1}
  \\
  \forall& q\in Q_h 
  &\;\,\exists &\disptest\in \Vdisp_h
  & \;:\quad&
  &\div\disptest&=q
  &\quad\wedge\quad&
  &\Vert \disptest \Vert_{1,h} &\le \frac1{\beta_{\Vdisp}}\Vert q \Vert_\Omega
  \label{eq:inf-sup2}
\end{alignat}

The semi-discrete scheme is: for all $t>0$ find $(p_h(t),\darcy_h(t),\disp_h(t))\in Q_h\times\Vdarcy_h\times\Vdisp_h$ such that
\begin{subequations}
  \label{eq:contscheme}
\begin{xalignat}2
  \bigl(\partial_t (\storage p_h+\biotwillis \div\disp_h),q\bigr)_\Omega
  +\bigl(\div \darcy_h,q\bigr)_\Omega
  &= \bigl(f_1,q\bigr)_\Omega,
  & \forall& q\in Q_h,
  \label{eq:contscheme1}
  \\
  \bigl(K^{-1}\darcy_h,\darcytest\bigr)_\Omega
  - \bigl(p_h,\div \darcytest\bigr)_\Omega
  &=- (p_{\mathrm{D}},\darcytest\cdot\bfn)_{\Gamma_{p\mathrm{D}}},
  &\forall&\darcytest\in\Vdarcy_h,
  \label{eq:contscheme2}
  \\
  a_h(\disp_h,\disptest)
  - \biotwillis  \bigl(p_h, \div \disptest\bigr)_\Omega
  &= \mathcal R(\disptest),
 &\forall& \disptest\in\Vdisp_h,
\label{eq:contscheme3}
\end{xalignat}
where
\begin{gather*}
  \mathcal R(\disptest) = 
   (\bff_2,\disptest)_\Omega
   + (\bfsigma_{\mathrm{N}},\disptest)_{\Gamma_{\bcdispN}}
  -2\mu (\DD(\disptest)\bfn,\disp_{\mathrm D})_{\Gamma_{\bcdispD}}
  +\frac{\gamma}{h} (\disp_{\mathrm D},\disptest)_{\Gamma_{\bcdispD}}.
\end{gather*} 
We have the following initial conditions for pressure $p_h(0)\in Q_h$ and for displacement $\disp_h(0)\in \Vdisp_h$
\begin{gather}
  \label{eq:initialdisp}
  \begin{alignedat}{2}
    (p_h(0),q) &= (p_0,q),\qquad&\forall& q\in Q_h,\\
    a_h(\disp_h(0),\disptest) &= a_h(\disp_0,\disptest),    
    \qquad& \forall& \disptest\in\Vdisp_h.
  \end{alignedat}
\end{gather}
\end{subequations}

We first note that the scheme (\ref{eq:contscheme1}--d) is consistent:
\begin{lemma}
  Let $(p,\disp,\darcy)$ be the solution to
  \eqref{eq:pdemass}-\eqref{eq:dispbcN}, and assume 
  $\disp(t)\in H^{3/2+\epsilon}(\Omega)$ for all $t$ and for some positive $\epsilon$. Then,
  it satisfies the equations
  \eqref{eq:contscheme1}-\eqref{eq:initialdisp}.
\end{lemma}
\begin{proof}
  The consistency of equation~\eqref{eq:contscheme} without the
  pressure term for solutions $u\in H^{3/2+\epsilon}(\Omega)$ of
  equation~\eqref{eq:momentum} was established in~\cite[Lemma
  2.1]{RiviereShawWheelerWhiteman2003}. Since
  $\Vdisp_h \subset \Hdiv(\Omega)$, the discretization of
  $(p_h, \div\disptest)$ is conforming. Thus, we obtain
  consistency of the momentum equation~\eqref{eq:contscheme3}
  with~\eqref{eq:momentum} and of the compatibility
  condition~\eqref{eq:initialdisp} with~\eqref{eq:initial}.

  The mixed finite element discretization (\ref{eq:contscheme}a--b) of
  \eqref{eq:pdemass}, \eqref{eq:pdedarcy}  is conforming and thus
  straightforward~\cite{BoffiBrezziFortin13}.
\end{proof}

We next state the coercivity of the bilinear form $d(\cdot,\cdot)$, the
proof of which depends on Korn's inequality for discontinuous
spaces~\cite{Brenner04Korn} and can be found in \cite{HansboLarson02}
\begin{lemma}
Assume $\gamma$ is large enough. There is a positive constant $\kappa$ independent of $h$ (and $\lambda, \mu, \biotwillis , \storage$) such that:
\begin{equation}\label{eq:acoer}
\kappa \Vert \disptest_h\Vert_{1,h}^2 \leq d_h(\disptest_h,\disptest_h),\quad
\forall \disptest_h\in\Vdisp_h.
\end{equation}
\label{lem:coerc}
\end{lemma}
The norm $\Vert\cdot\Vert_{1,h}$ is defined as:
\[
\Vert \disptest_h\Vert_{1,h}  = \left(\sum_{T\in\Th_h} \Vert \nabla \disptest_h\Vert_T^2 + \sum_{F\in\Gamma_i} \frac{\gamma}{h} \Vert [\disptest_h]\Vert_F^2 + \sum_{F\in\Gamma_{\bcdispD}} \frac{\gamma}{h}\Vert\disptest_h\Vert_F^2\right)^{1/2},\quad
\forall \disptest_h\in\Vdisp_h.
\]
As a corollary of Lemma~\ref{lem:coerc}, we have
\begin{equation}
\kappa \mu \Vert \disptest_h\Vert_{1,h}^2 +\lambda \Vert \nabla \cdot\disptest_h\Vert_\Omega^2 \leq a_h(\disptest_h,\disptest_h),\quad
\forall \disptest_h\in\Vdisp_h.
\end{equation}

We follow~\cite{yi2013coupling} and apply the theory of differential algebraic
equations to the solution of the semidiscrete problem. To this end, we
need the following lemma:

\begin{lemma}
  A  differential algebraic equation of the form
  \begin{gather*}
    \bfE \partial_t \bfx(t) + \bfA \bfx(t) = \bfq(t)
  \end{gather*}
  with $\bfA, \bfE\in \mathbb R^{m\times m}$ and $\bfq(t) \in \mathbb R^m$ is
  solvable, if and only if the matrix pencil $\sigma \bfE+\bfA$ is regular,
  that is, there is a value $\sigma \neq 0$, such that $\sigma \bfE+\bfA$
  is an invertible matrix.
\end{lemma}
This lemma can be found in~\cite[Theorem
2.4]{MattheijMolenaar02}. Solvable DAE have the property, that initial
value problems are uniquely solvable, if the initial condition is
compatible with the algebraic constraints. Thus, it remains to verify
that the semidiscrete system~(\ref{eq:contscheme}a--c) meets the
assumptions of this lemma. Obviously, these equations constitute a
finite dimensional linear system of equations with $\bfE$ corresponding
to the time derivative part. Thus, it remains to show the following lemma.

\begin{lemma}
  \label{lemma:uniqueness}
  \begin{subequations}
    \label{eq:uniqueness}
    For any $\sigma > 0$, the system
    \begin{xalignat}2
      \bigl(\sigma \bigl(\storage p_h+\biotwillis \div\disp_h),q\bigr) 
      +\bigl(\div \darcy_h,q\bigr) &= 0,
      & \forall& q\in Q_h,
      \label{eq:uniqueness1}
      \\
      \bigl(K^{-1}\darcy_h,\darcytest\bigr)
      - \bigl(p_h,\div \darcytest) &= 0,
      &\forall&\darcytest\in\Vdarcy_h,
      \label{eq:uniqueness2}
      \\
      a_h(\disp_h,\disptest\bigr)
      - \biotwillis  \bigl(p, \div \disptest\bigr) &= 0,
      &\forall& \disptest\in\Vdisp_h,
      \label{eq:uniqueness3}
    \end{xalignat}    
  \end{subequations}
  has the unique solution $(p_h, \darcy_h, \disp_h) = {\bf 0}$.
\end{lemma}

\begin{proof}
  Choosing test functions $q = p_h$, $\darcytest = \darcy_h$, and
  $\disptest = \sigma \disp_h$ and adding the three equations, we obtain
  \begin{gather*}
    \sigma \storage \Vert  p_h\Vert_\Omega^2
    +\Vert K^{-1/2}\darcy_h\Vert_\Omega^2 
    +\sigma\mu d_h(\disp_h,\disp_h)
    + \sigma\lambda \Vert \div \disp_h\Vert_\Omega^2
    = 0.
  \end{gather*}
  This, combined with the coercivity of $d_h(\cdot,\cdot)$, yields $\darcy_h=0$ and $\disp_h=0$ and concludes the proof for
  $\storage \neq 0$. For $\storage=0$, we choose in~\eqref{eq:uniqueness2}
  according to the inf-sup condition a test function
  $\darcytest \neq 0$ with $\div \darcytest = p_h$. Thus, $p_h=0$.
\end{proof}

Thus, together with the previous lemma, our DAE is solvable.  This
lemma indeed proved that there is not only one $\sigma$ for which the
problem is solvable, but that it is solvable for all positive
$\sigma$. While such a strong statement is not needed here, it is the
core of the proof of well-definedness of time stepping schemes below.

\section{A priori error estimates for continuous-in-time scheme}
\label{sec:priori-error-estim}

In this section, we state our theoretical results.  The proofs are
given in the rest of the paper. We begin with a simple lemma on math
conservation, which motivated us to choose this method. It turns out
that mass conservation is achieved pointwisely by this method.
\begin{lemma}
  \label{lemma:conservation}
  Let the spaces $Q_h$, $\Vdisp_h$, and $\Vdarcy_h$ be divergence
  conforming as in equations~\eqref{eq:pressurespace}
  and~\eqref{eq:Vhpressurecond}. Then, for any $\sigma>0$, the
  solution $(p_h,\darcy_h,\disp_h)$ of the
  system~\eqref{eq:uniqueness} obey the pointwise mass conservation
  equation
  \begin{gather}
    \label{eq:conservation}
    \sigma \bigl(\storage p_h+\biotwillis \div\disp_h) + \div \darcy_h = 0,
    \qquad \forall x\in T,\, \forall T\in \Th_h.
  \end{gather}
\end{lemma}

\begin{proof}
  We denote
  \begin{gather*}
    r_h = \sigma \bigl(\storage p_h+\biotwillis \div\disp_h) + \div \darcy_h.
  \end{gather*}
  Because of assumptions \eqref{eq:pressurespace} and \eqref{eq:Vhpressurecond}, the quantity $r_h$ belongs to $Q_h$.
  We test equation~\eqref{eq:uniqueness1} with $r_h$ to obtain the result.
\end{proof}

Next, we investigate the elastic subproblem. Let
$\tilde{\disp}(t) \in\Vdisp_h$ be the projection of $\disp(t)$ onto
$\Vdisp_h$ with respect to the linear elasticity operator, namely for
any $t>0$ let $\tilde{\disp}(t)$ satisfy
\begin{equation}
  \label{eq:wproj}
  a_h(\tilde{\disp}(t),\disptest) 
  = a_h(\disp(t),\disptest), \quad
  \forall \disptest\in\Vdisp_h.
\end{equation}
From the coercivity of $a_h(\cdot,\cdot)$, it is easy to see that $\tilde{\disp}(t)$
exists and is unique. By adaptation of~\cite[Theorem~8]{HansboLarson02}
to the Raviart-Thomas element and by the standard duality argument,
we have
\begin{prop}
  \label{thm:proj}
  There is a constant $C$ independent of
$h, \lambda, \mu, \biotwillis , \storage$ such that
\begin{gather}
  \label{eq:elasiticity-h1}
  \lVert \tilde{\disp}(t) -\disp(t)\rVert_{1,h}^2
  \leq C h^{2k} \lvert \disp(t)\rvert_{H^{k+1}(\Omega)}^2,
  \\    
  \lVert \tilde{\disp}(t) -\disp(t)\rVert_{\Omega}^2
  \leq C h^{2k+2} \lvert \disp(t)\rvert_{H^{k+1}(\Omega)}^2,
\\
\lVert \partial_t (\tilde{\disp}(t) -\disp(t))\rVert_{1,h}^2
  \leq C h^{2k} \lvert \partial_t\disp(t)\rvert_{H^{k+1}(\Omega)}^2.
\end{gather}
\end{prop}

We have furthermore observed in experiments, that the divergence is
converging optimally. These experiments included rectangular meshes
with local refinement. Currently, there is no proof for this fact, and
it may be due to superconvergence effects related to the meshes we
used. Following~\cite{ArnoldBoffiFalk05}, we do not expect this to
hold on general quadrilateral meshes. Nevertheless, we would like to
present an analysis using this fact alongside standard
convergence. Accordingly, we will use at some point:

\begin{assumption}
  \label{ass:hdiv}
There is a constant $C_{\mu,\lambda}$ that is independent of
$h, \biotwillis , \storage$ such that
\begin{gather}
  \label{eq:elasticity-hdiv}
  \lVert\div (\tilde{\disp}(t)-\disp(t))\rVert_{\Omega}^2
  \leq C_{\mu,\lambda} h^{2k+2} \lvert \div\disp(t)\rvert_{H^{k+1}(\Omega)}^2.
\end{gather}
\end{assumption}
This assumption would naturally imply 
\begin{gather}
  \label{eq:elasticity-hdiv-time}
  \lVert\div \partial_t(\tilde{\disp}(t)-\disp(t))\rVert_{\Omega}^2
  \leq C_{\mu,\lambda} h^{2k+2} \lvert \div\partial_t\disp(t)\rvert_{H^{k+1}(\Omega)}^2.
\end{gather}

Now we are ready to state our first main theorem:
\begin{theorem}
\label{thm:conv}
There is a constant $C$ independent of $h,  \lambda, \mu, \biotwillis , \storage$
such that
\begin{eqnarray*}
\forall t>0\quad
\mu \Vert \disp_h(t) -\disp(t)\Vert_{1,h}^2
\leq C h^{2k} (\mathcal{M} +\mu \Vert \disp(t)\Vert_{H^{k+1}(\Omega)}^2),
\end{eqnarray*}
and
\begin{eqnarray*}
\storage \Vert p_h(t) - p(t)\Vert_{\Omega}^2
+\Vert K^{-1/2} (\darcy_h-\darcy)\Vert_{L^2(0,t;L^2(\Omega))}^2
\leq C \epsilon_{\text{div}}(h)^2 (\mathcal{M}+
\storage \Vert p(t)\Vert_{H^{k+1}(\Omega)}^2),
\end{eqnarray*}
where $\epsilon_{\text{div}}(h) = h^k$ and
\[
\mathcal{M} = 
\biotwillis ^2\Vert \partial_t \div \disp\Vert_{L^2(0,T;H^{k+1}(\Omega))}^2
+\Vert \darcy\Vert_{L^2(0,T;H^{k+1}(\Omega))}^2.
\]
If in addition
Assumption~\ref{ass:hdiv} holds, we have
$\epsilon_{\text{div}}(h) =h^{k+1}$
and there is a constant $C_{\mu,\lambda}$ independent
of $h, \biotwillis , \storage$ such that
\begin{eqnarray*}
\lambda \Vert \div (\disp_h(t)-\disp(t))\Vert_{\Omega}^2
\leq C_{\mu,\lambda} h^{2k+2} (\mathcal{M} +\lambda\Vert\div\disp(t)\Vert_{H^{k+1}(\Omega)}^2).
\end{eqnarray*}
\end{theorem}

\subsection{Proof of Theorem~\ref{thm:conv}}
\label{sec:proof1}

We decompose the numerical error into an approximation error and a
discrete error.  For all $t>0$, choose $\tilde{\disp}(t)\in\Vdisp_h$
the $a_h(\cdot,\cdot)$-orthogonal projection of $\disp(t)$ satisfying
(\ref{eq:wproj}).  Denote the Fortin projection
$\tilde{\darcy}(t) = \pi_h \darcy(t) \in \Vdarcy_h$ and let
$\tilde{p}(t)$ be the $L^2$ projection of $p(t)$ in $Q_h$:
\begin{equation}\label{eq:L2proj}
\left(p(t)-\tilde{p}(t),q_h\right) = 0,\quad \forall q_h\in Q_h.
\end{equation}
This implies that 
\begin{equation}\label{eq:L2projder}
\left(\frac{\partial p}{\partial t}(t)-\frac{\partial\tilde{p}}{\partial t}(t),q_h\right) = 0,\quad \forall q_h\in Q_h.
\end{equation}
We have the following approximation error bound for the pressure:
\begin{eqnarray}
\Vert \tilde{p}(t)-p(t) \Vert_{\Omega} \leq C h^{k+1} \vert p(t) \vert_{H^{k+1}(\Omega)}.
\end{eqnarray}
Let us prove a lemma on the error $p_h-\tilde{p}$.
\begin{lemma}
There is a constant $C$ independent of $h, \mu, \lambda, \biotwillis , \storage$ such that
\begin{gather}
  \label{eq:pressure-by-inf-sup}
  \Vert p_h(t)-\tilde{p}(t)\Vert_\Omega
  \leq C \Vert \darcy_h(t)-\darcy(t)\Vert_\Omega,
  \quad\forall t>0.  
\end{gather}

\label{lem:pressure}
\end{lemma}
\begin{proof}
The error equation is
\[
(K^{-1}(\darcy_h-\darcy),\darcytest)_\Omega 
- (p_h-p,\div \darcytest)_\Omega = 0,\quad
\forall\darcytest\in\Vdarcy_h.
\]
Equivalently,
\[
 (p_h-\tilde{p},\div \darcytest)_\Omega = 
(K^{-1}(\darcy_h-\darcy),\darcytest)_\Omega 
 +(p-\tilde{p},\div \darcytest)_\Omega,\quad
\forall\darcytest\in\Vdarcy_h.
\]
Using properties (\ref{eq:L2proj}) and \eqref{eq:pressurespace}, we have 
\begin{equation}
 (p_h-\tilde{p},\div \darcytest)_\Omega = 
(K^{-1}(\darcy_h-\darcy),\darcytest)_\Omega
\le\Vert K^{-\frac12}(\darcy_h-\darcy)\Vert \,\Vert \darcytest\Vert
,\quad
\forall\darcytest\in\Vdarcy_h.
\label{eq:lemmerr}
\end{equation}
Since the pair $(\Vdarcy_h,Q_h)$ satisfies the inf-sup
condition~\eqref{eq:inf-sup1}, we can choose a test function
$\darcytest$ with $\div \darcytest = p_h -\tilde p$ and obtain
\begin{gather*}
  \Vert p_h - \tilde p\Vert^2
  \le \Vert K^{-\frac12}(\darcy_h-\darcy)\Vert
  \,\frac1{\beta_\Vdarcy}\Vert p_h - \tilde p\Vert,
\end{gather*}
which proves the result.
\end{proof}

We now write the system of error equations:
\begin{multline}
\bigl(
  \storage \partial_t(p_h-\tilde{p})
  +\biotwillis \partial_t\div(\disp_h-\tilde{\disp})
  + \div (\darcy_h-\tilde{\darcy}),q\bigr)
\\
  =\bigl(
  \storage \partial_t(p-\tilde{p})
  +\biotwillis \partial_t\div(\disp-\tilde{\disp})
  +\div (\darcy-\tilde{\darcy}),q\bigr),
\label{eq:conterror1}
\end{multline}
\begin{gather}
  \bigl(K^{-1}(\darcy_h-\tilde{\darcy}),\darcytest\bigr)
  - \bigl(p_h-\tilde{p},\div \darcytest\bigr) 
  =
  \bigl(K^{-1}(\darcy-\tilde{\darcy}),\darcytest\bigr)
  - \bigl(p-\tilde{p},\div \darcytest\bigr),
\label{eq:conterror2}  
\end{gather}
\begin{equation}
  a_h(\disp_h-\tilde{\disp},\disptest)
  - \biotwillis  \bigl(p_h-\tilde{p},\div \disptest\bigr)
  = a_h(\disp-\tilde{\disp},\disptest)
  - \biotwillis  \bigl(p-\tilde{p},\div \disptest\bigr).
\label{eq:conterror3}
\end{equation}
Next we  choose $q = p_h-\tilde{p}$, $\darcytest = \darcy_h-\tilde{\darcy}$ and $\disptest = \partial_t(\disp_h-\tilde{\disp})$ in \eqref{eq:conterror1}, \eqref{eq:conterror2} and \eqref{eq:conterror3} respectively.  We add the resulting equations and
obtain:
\begin{multline}
  \frac{\storage}{2}\frac{d}{dt} \bigl\Vert p_h-\tilde{p}\bigr\Vert_\Omega^2 
  + \bigl\Vert K^{-1/2}(\darcy_h-\tilde{\darcy})\bigr\Vert_\Omega^2
  +\frac{\mu}{2} \frac{d}{dt} d_h(\disp_h-\tilde{\disp},\disp_h-\tilde{\disp})
  + \frac{\lambda}{2} \frac{d}{dt} \bigl\Vert \div (\disp_h-\tilde{\disp})\bigr\Vert_\Omega^2
  \\
  =
  \bigl( \storage \partial_t(p-\tilde{p})
    +\biotwillis \partial_t\div(\disp-\tilde{\disp}),
    p_h-\tilde{p}\bigr)
  +\bigl(\div (\darcy-\tilde{\darcy}),p_h-\tilde{p}\bigr)
  \\
  +
  \bigl(K^{-1}(\darcy-\tilde{\darcy}),\darcy_h-\tilde{\darcy}\bigr)
  - \bigl(p-\tilde{p},\div (\darcy_h-\tilde{\darcy})\bigr)
\\
+ a_h(\disp-\tilde{\disp},\partial_t(\disp_h-\tilde{\disp})) 
- \biotwillis  (p-\tilde{p}, \div \partial_t(\disp_h-\tilde{\disp})).
\label{eq:add3}
\end{multline}
Using (\ref{eq:wproj}) we have
\[
 a_h(\disp-\tilde{\disp},\partial_t(\disp_h-\tilde{\disp})) = 0.
\]
Using property~\eqref{eq:prop1} of the Fortin interpolation, we have
\[
(\div (\darcy-\tilde{\darcy}),p_h-\tilde{p}) = 0.
\]
Using properties (\ref{eq:L2proj}), \eqref{eq:pressurespace} and \eqref{eq:Vhpressurecond}
we have
\[
(p-\tilde{p},\div (\darcy_h-\tilde{\darcy}))=0,
\]
\[
\biotwillis  (p-\tilde{p}, \div \partial_t(\disp_h-\tilde{\disp})) = 0.
\]
Using property \eqref{eq:L2projder}, we have
\[
\bigl( \storage \partial_t(p-\tilde{p}),p_h-\tilde{p}\bigr) = 0.
\]
Thus, equation~\eqref{eq:add3} reduces to
\begin{multline*}
  \frac{\storage}{2}\frac{d}{dt} \Vert p_h-\tilde{p}\Vert_\Omega^2 
  + \Vert K^{-1/2}(\darcy_h-\tilde{\darcy})\Vert_\Omega^2
  +\frac{\mu}{2} \frac{d}{dt} d_h(\disp_h-\tilde{\disp},\disp_h-\tilde{\disp})
  + \frac{\lambda}{2} \frac{d}{dt} \Vert \div (\disp_h-\tilde{\disp})\Vert^2
  \\
  =
  \bigl(\biotwillis \partial_t\div(\disp-\tilde{\disp}),p_h-\tilde{p}\bigr)
  + \bigl(K^{-1}(\darcy-\tilde{\darcy}),\darcy_h-\tilde{\darcy}\bigr).
\end{multline*}
The second term on the right-hand side is easily bounded
by approximation bounds
\[
(K^{-1}(\darcy-\tilde{\darcy}),\darcy_h-\tilde{\darcy})\leq
  \Vert K^{-1/2} (\darcy-\tilde{\darcy})\Vert_\Omega^2 
+ \frac14 \Vert K^{-1/2} (\darcy_h-\tilde{\darcy})\Vert_\Omega^2.
\]
\[
\leq C h^{2k+2} \Vert \darcy\Vert_{H^{k+1}(\Omega)}^2
+ \frac14 \Vert K^{-1/2} (\darcy_h-\tilde{\darcy})\Vert_\Omega^2.
\]
For the first term in the right-hand side we use Lemma~\ref{lem:pressure}
\begin{gather*}
\bigl(\biotwillis \partial_t\div(\disp-\tilde{\disp}),p_h-\tilde{p}\bigr)
\leq C \bigl\Vert \biotwillis \partial_t\div(\disp-\tilde{\disp})\bigr\Vert_\Omega
\bigl\Vert \darcy-\darcy_h\bigr\Vert_\Omega\\
\leq C \bigl\Vert \biotwillis \partial_t\div(\disp-\tilde{\disp})\bigr\Vert_\Omega
(\bigl\Vert \darcy-\tilde{\darcy}\bigr\Vert_\Omega
+\bigl\Vert \tilde{\darcy}-\darcy_h\bigr\Vert_\Omega),
\end{gather*}
which yields
with approximation results
\begin{multline*}
\bigl(\biotwillis \partial_t\div(\disp-\tilde{\disp}),p_h-\tilde{p}\bigr)
\\
\leq C \biotwillis ^2 \Vert \partial_t \nabla \cdot (\disp-\tilde{\disp})\Vert_{\Omega}^2
+ C h^{2k+2} \Vert \darcy\Vert_{H^{k+1}(\Omega)}^2
+ \frac14 \Vert K^{-1/2}(\darcy_h-\tilde{\darcy})\Vert_\Omega^2.
\end{multline*}
Therefore the error bound becomes 
\begin{multline*}
\frac{\storage}{2}\frac{d}{dt} \Vert p_h-\tilde{p}\Vert_\Omega^2 
+ \frac12 \Vert K^{-1/2}(\darcy_h-\tilde{\darcy})\Vert_\Omega^2
+\frac{\mu}{2} \frac{d}{dt} d_h(\disp_h-\tilde{\disp},\disp_h-\tilde{\disp})
+ \frac{\lambda}{2} \frac{d}{dt} \Vert \div (\disp_h-\tilde{\disp})\Vert_\Omega^2
\\
\leq C \biotwillis ^2 \Vert \partial_t \nabla \cdot (\disp-\tilde{\disp})\Vert_{\Omega}^2
+ C h^{2k+2} \Vert \darcy\Vert_{H^{k+1}(\Omega)}^2.
\end{multline*}
Multiply by $2$,  integrate from $\tau=0$ to $\tau = t$ and remark that 
$p_h(0)=\tilde{p}(0)$ and $\disp_h(0)=\tilde{\disp}(0)$:
\begin{eqnarray*}
\storage\Vert p_h-\tilde{p}\Vert_\Omega^2 
+ \int_0^t \Vert K^{-1/2}(\darcy_h-\tilde{\darcy})\Vert_\Omega^2 d\tau
+ \mu d_h(\disp_h-\tilde{\disp},\disp_h-\tilde{\disp})
+ \lambda \Vert \div (\disp_h-\tilde{\disp})\Vert_\Omega^2
\\
\leq C \biotwillis ^2  \int_0^t \Vert \partial_t \nabla \cdot (\disp-\tilde{\disp})\Vert_{\Omega}^2 d\tau
+ C h^{2k+2} \int_0^t \Vert \darcy\Vert_{H^{k+1}(\Omega)}^2 d\tau.
\end{eqnarray*}
Thus we have using (\ref{eq:acoer}) 
\begin{eqnarray*}
\storage\Vert p_h-\tilde{p}\Vert_\Omega^2 
+ \int_0^t \Vert K^{-1/2}(\darcy_h-\tilde{\darcy})\Vert_\Omega^2 d\tau
+ \kappa\mu \Vert \disp_h-\tilde{\disp}\Vert_{1,h}^2 
+ \lambda \Vert \div (\disp_h-\tilde{\disp})\Vert_\Omega^2
\\
\leq C \biotwillis^2 \int_0^t \Vert \partial_t \nabla \cdot (\disp-\tilde{\disp})\Vert_{\Omega}^2 d\tau
+ C h^{2k+2} \int_0^t \Vert \darcy\Vert_{H^{k+1}(\Omega)}^2 d\tau.
\end{eqnarray*}
We then conclude using \eqref{eq:elasiticity-h1} or Assumption
\eqref{eq:elasticity-hdiv-time},
triangle inequalities and approximation
bounds. In the case $\storage=0$, the estimate above does not yield an
estimate for the pressure. This can be recovered by
Lemma~\ref{lem:pressure}, such that in addition to the estimate above,
the pressure is bounded by~\eqref{eq:pressure-by-inf-sup}.

\section{Discrete-in-time Scheme}
\label{sec:discrete-time-scheme}

Let $\Delta t > 0$ denote the time step, and define $t^n = n \Delta t$ for $n\in\mathbb{N}$. We use a first order in time Euler scheme
and seek $(p_h^{n+1},\darcy_h^{n+1},\disp_h^{n+1})\in Q_h\times\Vdarcy_h\times\Vdisp_h$ such that for all $n\geq 0$
\begin{subequations}
  \label{eq:disc}
  \begin{xalignat}2
    \left(\frac{1}{\Delta t} \bigl(\storage p_h^{n+1}+\biotwillis \div\disp_h^{n+1}\bigr),q\right)
    + \bigl(\div \darcy_h^{n+1},q\bigr)
    &= \mathcal R_p^{n+1}(q),
    & \forall& q\in Q_h
    \label{eq:disc-press}
    \\
    \bigl(K^{-1}\darcy_h^{n+1},\darcytest\bigr)
    - \bigl(p_h^{n+1},\div \darcytest\bigr) &=
    \bigl(p_D^{n+1},\darcytest\cdot\bfn\bigr)_{\Gamma_{p\mathrm{D}}},
    &\forall&\darcytest\in\Vdarcy_h,
    \label{eq:disc-flow}
    \\
    a_h(\disp_h^{n+1},\disptest)
    - \biotwillis  \bigl(p_h^{n+1}, \div \disptest\bigr) 
    &= \mathcal R_u^{n+1}(\disptest)
    &\forall& \disptest\in\Vdisp_h,
    \label{eq:disc-stress}
  \end{xalignat}
  where the linear functions in the right-hand sides are
  \begin{align*}
    \mathcal R_p^{n+1}(q) &= (f_1^{n+1},q) + \left(\frac{1}{\Delta t} \bigl(\storage p_h^n+\biotwillis \div\disp_h^n\bigr),q\right),
    \\
    \mathcal R_u^{n+1}(\disptest) &=
    (\bff_2^{n+1},\disptest) + (\bfsigma_N^{n+1},\disptest)_{\Gamma_{\bcdispN}}
    -2\mu (\DD(\disptest)\bfn,\disp_D^{n+1})_{\Gamma_{\bcdispD}}
    +\frac{\gamma}{h} (\disp_D^{n+1},  \disptest_{\Gamma_{\bcdispD}}),
  \end{align*}
\end{subequations}

with initial conditions:
\begin{xalignat}2
  (p_h^0,q) &= (p_0,q) &\forall q&\in Q_h,
  \label{eq:discIC1}
  \\
  a_h(\disp_h^0,\disptest)
  &= a_h(\disp_0,\disptest),
  &\forall \disptest&\in\Vdisp_h. \label{eq:discIC2}
\end{xalignat}

The short-hand notation $\disp_D^n, f_1^n, \bff_2^n$ and
$\bfsigma_N^n$ is used for the functions $\disp_D, f_1, \bff_2$ and
$\bfsigma$ evaluated at $t^n$.

\begin{lemma}[Existence and uniqueness]
  There exists an unique solution $p_h^n, \darcy_h^n, \disp_h^n$
  satisfying (\ref{eq:disc}a--c) for all $n\geq 0$.
\end{lemma}
\begin{proof}
  The proof follows closely the proof for existence and uniqueness in
  the semi-discrete section.  First, we note that the discrete initial
  conditions are the same as~\eqref{eq:initialdisp} and thus
  compatible with the momentum equation~\eqref{eq:disc-stress} at
  $t^0$.

  Assume now the solution
  at time $t_n$, $n\geq 0$ has been computed.  Since the problem
  (\ref{eq:disc}a--c) is linear and finite dimensional, it
  suffices to show uniqueness. Thus, assume the right hand side in
  (\ref{eq:disc}a--c)is zero. Then, we have the situation of
  Lemma~\ref{lemma:uniqueness} with $\sigma=1/\Delta t$ in
  equations~(\ref{eq:uniqueness}a--c). Thus, $p^{n+1}$,
  $\darcy^{n+1}$, and $\disp^{n+1}$ are well-defined.
\end{proof}


\begin{theorem}
  Let Assumption~\ref{ass:hdiv} hold. Then, 
there is a constant $C$ independent of $h, \mu, \lambda, \biotwillis , \storage$ such that
for all $m\geq 1$
\begin{align}
  \storage \Vert p_h^m-p(t^m)\Vert_\Omega^2
  &\leq C  h^{2k+2} \left(\mathcal M_h^2 + \storage \Vert p(t^m)\Vert_{H^{k+1}(\Omega)}^2\right)
    +  C \Delta t^2 \mathcal M_t^2,
  \\
  \mu \Vert \disp_h^{m}-\disp(t^m)\Vert_{1,h}^2
  &\leq C h^{2k+2} \mathcal M_h^2
    +  C \Delta t^2 \mathcal M_t^2
    +\mu C h^{2k}\Vert \disp(t^m)\Vert_{H^{k+1}(\Omega)}^2,
  \\
    \lambda \Vert \div (\disp_h^{m}-\disp(t^m))\Vert_\Omega^2
    &\leq C h^{2k+2} \mathcal M_h^2
      +  C \Delta t^2 \mathcal M_t^2
      +\lambda C_{\lambda,\mu} h^{2k+2} \Vert \nabla \cdot \disp(t^m)\Vert_{H^{k+1}(\Omega)}^2, 
\end{align}
\begin{eqnarray}
\Delta t \sum_{n=0}^{m-1}\Vert K^{-1/2}(\darcy_h^{n+1}-\darcy(t^{n+1}))\Vert_\Omega^2
\leq C h^{2k+2} \mathcal M_h^2
+  C \Delta t^2 \mathcal M_t^2.
\end{eqnarray}
where
\begin{align*}
  \mathcal M_h^2 &= \biotwillis ^2 \Vert \partial_t \disp\Vert_{L^2(0,T;H^{k+1}(\Omega))}^2
  + \Vert \darcy\Vert_{L^2(0,T;H^{k+1}(\Omega))}^2
  \\
  \mathcal M_t^2 &= \storage^2 \Vert p_{tt}\Vert_{L^\infty(0,T;L^2(\Omega))}^2 
                   +\biotwillis ^2\Vert \disp_{tt}\Vert_{L^\infty(0,T;L^2(\Omega))}^2
\end{align*}
\end{theorem}

Note that the theorem holds even without Assumption~\eqref{ass:hdiv}, but with reduced
convergence orders for $\darcy$ and $p$, similarly to 
Theorem~\ref{thm:conv}.

\begin{proof}
Error analysis follows closely the one at the continuous-in-time level. 
We can choose $\tilde{\darcy}^{n}$ such that
\[
\tilde{\darcy}^{n} = \pi_h \darcy(t^n), \quad n\geq 1
\]
%
We also can choose $\tilde{\disp}^{n}$ such that
\begin{equation}\label{eq:wprojdisc}
a_h(\tilde{\disp}^n,\disptest)
=  a_h(\disp(t^n),\disptest),
\quad\forall \disptest\in\Vdisp_h,\quad \forall n\geq 0
\end{equation}
Using Proposition~\ref{thm:proj} and Assumption~\ref{ass:hdiv}, we
have the following a priori error bounds, for any $n\geq 0$:
\begin{equation}
\Vert \tilde{\disp}^n -\disp(t^n)\Vert_{1,h}^2
\leq C h^{2k} \Vert \disp(t^n)\Vert_{H^{k+1}(\Omega)}^2,
\label{eq:discproj-energyn}
\end{equation}
\begin{equation}
\Vert \tilde{\disp}^n - \disp(t^n)\Vert_{\Omega}^2 + \Vert\div (\tilde{\disp}^n-\disp(t^n))\Vert_{\Omega}^2
\leq C_{\mu,\lambda} h^{2k+2} \Vert \disp(t^n)\Vert_{H^{k+1}(\Omega)}^2.
\label{eq:discproj-ltwon}
\end{equation}

We can also choose $\tilde{p}$ to be the $L^2$ projection of $p$ in $Q_h$:
\begin{equation}\label{eq:disL2proj}
(p(t^{n})-\tilde{p}^{n},q_h) = 0,\quad \forall q_h\in Q_h, \quad \forall n\geq 0.
\end{equation}
We decompose the errors as follows:
\begin{eqnarray*}
\darcy_h^n-\darcy(t^n) = \bfchi_\darcy^n - \bfeta_\darcy^n,\quad \bfchi_\darcy^n = \darcy_h^n-\tilde{\darcy}^n,\quad \bfeta_\darcy^n = \darcy(t^n) - \tilde{\darcy}^n,
\\
\disp_h^n-\disp(t^n) = \bfchi_\disp^n - \bfeta_\disp^n,\quad \bfchi_\disp^n = \disp_h^n-\tilde{\disp}^n,\quad \bfeta_\disp^n = \disp(t^n) - \tilde{\disp}^n,
\\
p_h^n-p(t^n) = \chi_p^n - \eta_p^n,\quad \chi_p^n = p_h^n-\tilde{p}^n,\quad \eta_p^n = p(t^n) - \tilde{p}^n.
\end{eqnarray*}
Using Taylor approximation, we have
\[
\frac{p(t^{n+1})-p(t^n)}{\Delta t} = \frac{\partial p}{\partial t} (t^{n+1})
+\Delta t \rho_{p,n+1},
\]
and
\[
\frac{\disp(t^{n+1})-\disp(t^n)}{\Delta t} = \frac{\partial \disp}{\partial t} (t^{n+1})
+\Delta t \rho_{\disp,n+1}.
\]
with
\begin{equation}
\Vert \rho_{p,n+1}\Vert_\Omega \leq C \Vert p_{tt}\Vert_{L^\infty(t_n,t_{n+1};L^2(\Omega))},\quad
\Vert \rho_{\disp,n+1}\Vert_\Omega \leq C \Vert \disp_{tt}\Vert_{L^\infty(t_n,t_{n+1};L^2(\Omega))}
\label{eq:taylorbounds}
\end{equation}

Error equations become:
\begin{eqnarray}
(\frac{1}{\Delta t} \left(\storage (\chi_p^{n+1}-\chi_p^n)+\biotwillis \div(\bfchi_\disp^{n+1}-\bfchi_\disp^n)\right),q)
+(\div \bfchi_\darcy^{n+1},q) = 
\nonumber
\\
(\frac{1}{\Delta t} \left(\storage (\eta_p^{n+1}-\eta_p^n)+\biotwillis \div(\bfeta_\disp^{n+1}-\bfeta_\disp^n)\right),q)_\Omega 
+(\div \bfeta_\darcy^{n+1},q) 
\nonumber
\\
+\Delta t (\storage \rho_{p,n+1}+\biotwillis \rho_{\disp,n+1},q),
\quad \forall q\in Q_h,
\label{eq:errordisc1}
\\
(K^{-1}\bfchi_\darcy^{n+1},\darcytest)- (\chi_p^{n+1},\div \darcytest) 
=
(K^{-1}\bfeta_\darcy^{n+1},\darcytest)- (\eta_p^{n+1},\div \darcytest),
\quad \forall\darcytest\in\Vdarcy_h,
\label{eq:errordisc2}
\\
a_h(\bfchi_\disp^{n+1},\disptest)
-(\biotwillis  \chi_p^{n+1}, \div \disptest)
= 
a_h(\bfeta_\disp^{n+1},\disptest)
- (\biotwillis  \eta_p^{n+1}, \div \disptest),
\quad \forall \disptest\in\Vdisp_h.
\label{eq:errordisc3}
\end{eqnarray}
We choose $q = \chi_p^{n+1}$ in \eqref{eq:errordisc1} and $\darcytest = \bfchi_\darcy^{n+1}$ in 
\eqref{eq:errordisc2}, and add the two resulting equations:
\begin{eqnarray}
(\frac{1}{\Delta t} \left(\storage (\chi_p^{n+1}-\chi_p^n)+\biotwillis \div(\bfchi_\disp^{n+1}-\bfchi_\disp^n)\right),\chi_p^{n+1})
+\Vert K^{-1/2}\bfchi_\darcy^{n+1}\Vert_\Omega^2
\nonumber\\
=(\frac{1}{\Delta t} \left(\storage (\eta_p^{n+1}-\eta_p^n)+\biotwillis \div(\bfeta_\disp^{n+1}-\bfeta_\disp^n)\right),\chi_p^{n+1})_\Omega 
+(\div \bfeta_\darcy^{n+1},\chi_p^{n+1}) 
\nonumber\\
+
(K^{-1}\bfeta_\darcy^{n+1},\bfchi_\darcy^{n+1})- (\eta_p^{n+1},\div \bfchi_\darcy^{n+1}) 
+\Delta t (\storage\rho_{p,n+1}+\biotwillis \rho_{\disp,n+1},\chi_p^{n+1}).
\label{eq:eq1eq2}
\end{eqnarray}
Next we select the test function $\disptest$ in \eqref{eq:errordisc3}
\[
\disptest = \frac{1}{\Delta t} (\bfchi_\disp^{n+1}-\bfchi_\disp^n),
\]
and add the resulting equation to \eqref{eq:eq1eq2}:
\begin{eqnarray}
(\frac{1}{\Delta t} \left(\storage (\chi_p^{n+1}-\chi_p^n)\right),\chi_p^{n+1})
&+&\Vert K^{-1/2}\bfchi_\darcy^{n+1}\Vert_\Omega^2
+\frac{1}{\Delta t}  a_h(\bfchi_\disp^{n+1},\bfchi_\disp^{n+1}-\bfchi_\disp^n) 
\nonumber\\
&=&(\frac{1}{\Delta t} \storage (\eta_p^{n+1}-\eta_p^n), \chi_p^{n+1})
+(\frac{1}{\Delta t} \biotwillis \div(\bfeta_\disp^{n+1}-\bfeta_\disp^n),\chi_p^{n+1}) 
\nonumber\\
+(\div \bfeta_\darcy^{n+1},\chi_p^{n+1}) 
&+&
(K^{-1}\bfeta_\darcy^{n+1},\bfchi_\darcy^{n+1})- (\eta_p^{n+1},\div \bfchi_\darcy^{n+1}) 
+\Delta t (\storage\rho_{p,n+1}+\biotwillis \rho_{\disp,n+1},\chi_p^{n+1})
\nonumber\\
&+&\frac{1}{\Delta t}  a_h(\bfeta_\disp^{n+1},\bfchi_\disp^{n+1}-\bfchi_\disp^n) 
- (\biotwillis  \eta_p^{n+1}, \div \frac{1}{\Delta t} (\bfchi_\disp^{n+1}-\bfchi_\disp^n))\nonumber\\
&=& T_1+\dots + T_8.
\label{eq:eq3eq4}
\end{eqnarray}
Because of \eqref{eq:pressurespace} and the definition of the $L^2$ projection (see \eqref{eq:disL2proj}), the  terms
$T_1, T_5$ and $T_8$ vanish. Because of \eqref{eq:prop1}, the term $T_3$ is zero. Finally, because of  \eqref{eq:wprojdisc}, 
the term $T_7$ also vanishes.
Therefore \eqref{eq:eq3eq4} simplifies to:
\begin{eqnarray}
(\frac{1}{\Delta t} \left(\storage (\chi_p^{n+1}-\chi_p^n)\right),\chi_p^{n+1})
+\Vert K^{-1/2}\bfchi_\darcy^{n+1}\Vert_\Omega^2
+\frac{1}{\Delta t}  a_h(\bfchi_\disp^{n+1},\bfchi_\disp^{n+1}-\bfchi_\disp^n) 
\nonumber\\
=\frac{\biotwillis}{\Delta t} (\div(\bfeta_\disp^{n+1}-\bfeta_\disp^n),\chi_p^{n+1})
+
(K^{-1}\bfeta_\darcy^{n+1},\bfchi_\darcy^{n+1})
+\Delta t (\storage\rho_{p,n+1}+\biotwillis \rho_{\disp,n+1},\chi_p^{n+1}).
\end{eqnarray}
Lemma~\ref{lem:pressure} is valid at the discrete level:
\[
\Vert p_h^n-\tilde{p}^n\Vert_\Omega \leq C \Vert \darcy_h^n-\darcy(t^n)\Vert_\Omega,\quad
\forall n \geq 1.
\]
This means that
\[
\Vert \chi_p^n \Vert_\Omega \leq C (\Vert \bfchi_\darcy^n\Vert_\Omega + \Vert \bfeta_\darcy^n\Vert_\Omega),\quad
\forall n\geq 1.
\]
Therefore this implies
\begin{eqnarray*}
\frac{\storage}{2\Delta t} (\Vert \chi_p^{n+1}\Vert_\Omega^2-\Vert \chi_p^n\Vert_\Omega^2)
+\frac12 \Vert K^{-1/2}\bfchi_\darcy^{n+1}\Vert_\Omega^2
+\frac{\kappa}{2\Delta t} \mu (\Vert \bfchi_\disp^{n+1}\Vert_{1,h}^2-\Vert \bfchi_\disp^n\Vert_{1,h}^2) 
\nonumber\\
+\lambda \frac{1}{2\Delta t} (\Vert \div \bfchi_\disp^{n+1}\Vert_\Omega^2 - \Vert \div \bfchi_\disp^n\Vert_\Omega^2)
\nonumber\\
\leq  C \Vert \bfeta_\darcy^{n+1}\Vert_\Omega^2 
+ C \frac{\biotwillis ^2}{\Delta t^2} \Vert \div(\bfeta_\disp^{n+1}-\bfeta_\disp^n)\Vert_\Omega^2
+ C  \Delta t^2 \Vert \storage \rho_{p,n+1}+\biotwillis \rho_{\disp,n+1}\Vert_\Omega^2.
\end{eqnarray*}
We multiply the above inequality by $2\Delta t$,  sum from $n=0$ to $n=m-1$ and
remark that $\chi_p^0 = 0$ and $\bfchi_\disp^0={\bf 0}$:
\begin{eqnarray}
\storage \Vert \chi_p^{m}\Vert^2
+\Delta t \sum_{n=0}^{m-1}\Vert K^{-1/2}\bfchi_\darcy^{n+1}\Vert^2
+\kappa \mu \Vert \bfchi_\disp^{m}\Vert_{1,h}^2
+\lambda \Vert \div \bfchi_\disp^{m}\Vert_\Omega^2
\nonumber\\
\leq  
  C \biotwillis ^2\Delta t\sum_{n=0}^{m-1} \Vert \frac{1}{\Delta t}\div(\bfeta_\disp^{n+1}-\bfeta_\disp^n)\Vert_\Omega^2
+ C \Delta t\sum_{n=0}^{m-1} \Vert \bfeta_\darcy^{n+1}\Vert_\Omega^2
\nonumber\\
+  \Delta t^3\sum_{n=0}^{m-1} \Vert \storage\rho_{p,n+1}+\biotwillis \rho_{\disp,n+1}\Vert_\Omega^2.
\end{eqnarray}
From the approximation bounds \eqref{eq:discproj-ltwon}, we have 
\[
\Delta t\sum_{n=0}^{m-1} \Vert \frac{1}{\Delta t}\div(\bfeta_\disp^{n+1}-\bfeta_\disp^n)\Vert_\Omega^2
\leq C h^{2k+2} \Delta t \sum_{n=0}^{m-1}\Vert \frac{\disp(t^{n+1})-\disp(t^n)}{\Delta t} \Vert_{H^{k+1}(\Omega)}^2
\]
\[
\leq C h^{2k+2} \Delta t \sum_{n=0}^{m-1}\Vert \frac{\partial\disp}{\partial t}(t^{\ast,n}) \Vert_{H^{k+1}(\Omega)}^2
\leq C h^{2k+2} \Vert \frac{\partial\disp}{\partial t} \Vert_{L^2(0,T;H^{k+1}(\Omega))}^2.
\]
Similarly we obtain
\[\Delta t\sum_{n=0}^{m-1} \Vert \bfeta_\darcy^{n+1}\Vert_\Omega^2
\leq C h^{2k+2} \Delta t \sum_{n=0}^{m-1} \Vert \darcy(t^{n+1})\Vert_{H^{k+1}(\Omega)}^2
\leq C h^{2k+2} \Vert \darcy\Vert_{L^2(0,T;H^{k+1}(\Omega))}^2.
\]
Finally using \eqref{eq:taylorbounds}, we obtain:
\begin{eqnarray}
\storage \Vert \chi_p^{m}\Vert_\Omega^2
+\Delta t \sum_{n=0}^{m-1}\Vert K^{-1/2}\bfchi_\darcy^{n+1}\Vert_\Omega^2
+\kappa \mu \Vert \bfchi_\disp^{m}\Vert_{1,h}^2
+\lambda \Vert \div \bfchi_\disp^{m}\Vert_\Omega^2
\nonumber\\
\leq C \biotwillis ^2 h^{2k+2} \Vert \frac{\partial\disp}{\partial t} \Vert_{L^2(0,T;H^{k+1}(\Omega))}^2
+ C h^{2k+2} \Vert \darcy\Vert_{L^2(0,T;H^{k+1}(\Omega))}^2
\nonumber\\
+  C \Delta t^2 (\storage^2 \Vert p_{tt}\Vert_{L^\infty(0,T;L^2(\Omega)}^2 
+\biotwillis ^2\Vert \disp_{tt}\Vert_{L^\infty(0,T;L^2(\Omega))}^2).
\end{eqnarray}
The final results are obtained by triangle inequalities and approximation bounds.
\end{proof}

\begin{remark}
Let $\tilde{f}_1$ denote the $L^2$ projection of $f_1$ onto $Q_h$. The discrete pressure and displacement satisfy the conservation property, pointwisely:
\[
    \frac{1}{\Delta t} \bigl(\storage p_h^{n+1}+\biotwillis \div\disp_h^{n+1}\bigr)
    -\frac{1}{\Delta t} \bigl(\storage p_h^{n}+\biotwillis \div\disp_h^{n}\bigr)
= \tilde{f}_1^{n+1}, \quad\forall x\in T,\quad\forall T\in\Th_h.
\]

\end{remark}

\section{Numerical experiments}
\label{sec:numer-exper}

For our numerical experiments, we follow the approach
in~\cite{BarryMercer1998} to construct an exact solution to
equations~\eqref{eq:pdemass}--\eqref{eq:momentum}. Differing from
their results, we construct smooth solutions in order to verify the
expected convergence orders. To this end, we let $\mu=1$ and
$K=\mathbf I$, which corresponds to the nondimensionalization
in~\cite{BarryMercer1998} and does not
restrict generality of our results. Furthermore, we consider only
incompressible fluids, that is, $\storage = 0$. Further, we take
$\biotwillis = 1$. We choose $\Omega = (0,1)^2$ with boundary conditions
\begin{gather}
  \left.
    \arraycolsep2pt
    \begin{matrix}
      \partial_n (\disp\cdot\bfn) &=&0 \\
       \disp\times \bfn &=&0 \\
      p &=&0
    \end{matrix}\;
  \right\} \text{ on } \partial\Omega.
\end{gather}
Thus, the deformation can only be in normal direction on each
boundary, and the pressure is prescribed. The seepage velocity at the boundary is free.  Let
$\phi(x,y) = \sin(2\pi x)\sin(2\pi y)$ and choose as right hand side
in~\eqref{eq:pdemass}
\begin{gather}
  \label{eq:f1-barry}
  f_1(x,y,t) = \phi(x,y) \sin(2\pi t).
\end{gather}
With the auxiliary function
\begin{gather}
  \psi(t) = \frac{1}{64 \pi^4+4\pi^2}
  \bigl(8\pi^2 \sin(2\pi t) - 2\pi \cos(2\pi t)
  + 2\pi e^{-8\pi^2 t} \bigr),
\end{gather}
we obtain the solutions
\begin{gather}
  \begin{split}
    p(x,y,t) & = \psi(t)\phi(x,y), \\
    \darcy(x,y,t) &= \psi(t) \nabla \phi(x,y), \\
    \disp(x,y,t) &= \frac{\psi(t)}{8\pi^2} \nabla \phi(x,y).
  \end{split}
\end{gather}

We discretize $\Omega$ by a sequence of  Cartesian meshes, such
that $\Th_0$ is the mesh consisting of the single square
$\Omega$. The mesh $\Th_{\ell}$ is defined recursively by
dividing every square of $\Th_{\ell-1}$ into four congruent
squares. Thus, $\Th_{\ell}$ consists of $4^\ell$ mesh cells with
sides of length $2^{-\ell}$. Figure~\ref{fig:warped} shows the
solution at time $t=0.5$ with considerably enlarged deformations and
seepage velocity arrows.

\begin{figure}[tp]
  \centering
  \includegraphics[width=\textwidth]{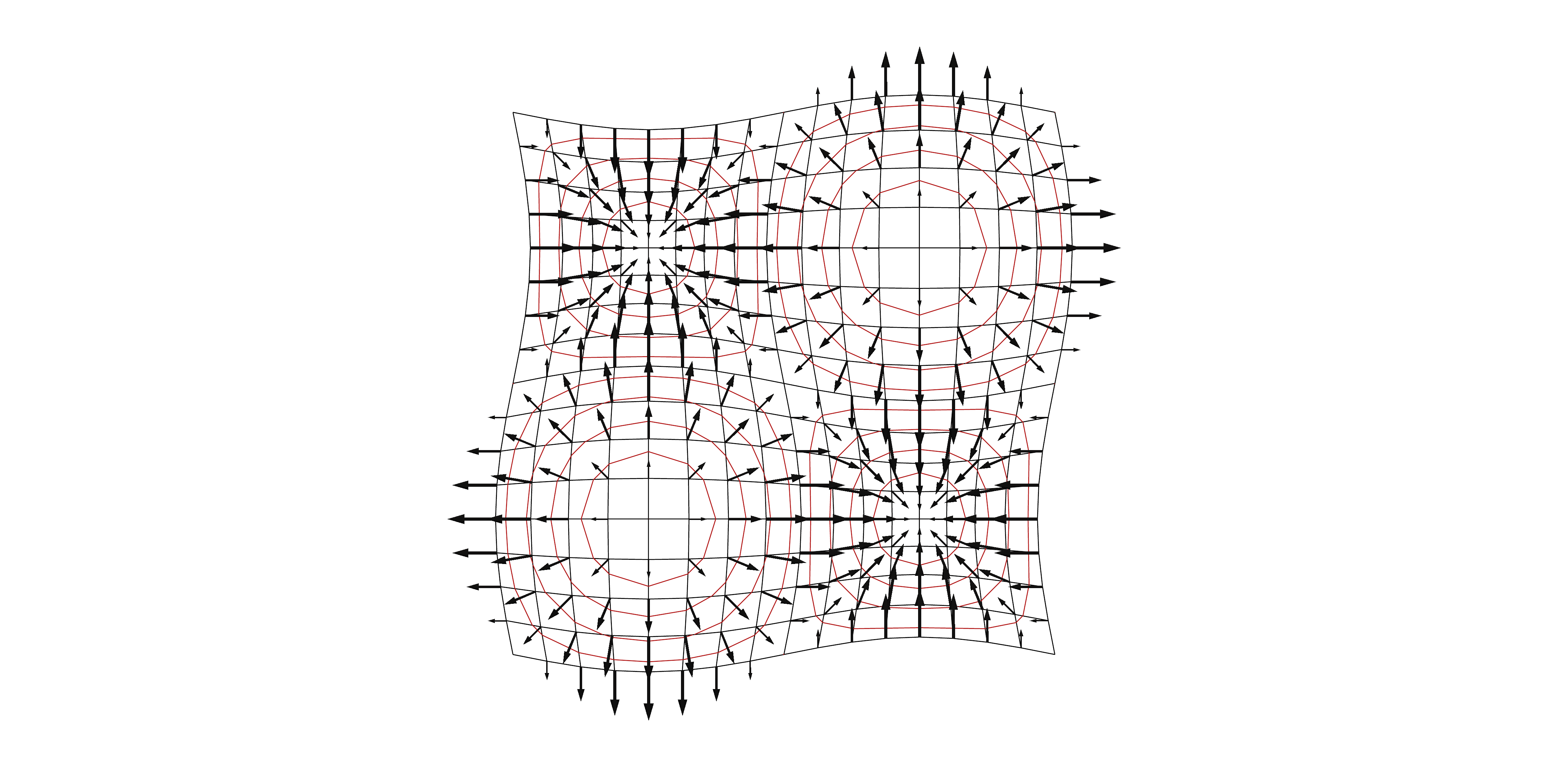}
  \caption{The seepage velocity $\darcy$ (arrows) and the pressure $p$ (isolines) on the mesh deformed by $\disp$ (arrows and deformations not in scale)}
  \label{fig:warped}
\end{figure}

In Figure~\ref{fig:errors-rt1q1}, we display different norms of the
errors of $\disp$, $\darcy$, and $p$, respectively. Note that all
$L^2$-errors as well as the quadratic errors of the divergences are of
second order, while the errors of the gradients are first order,
confirming our theoretical results  and the assumption on the divergence error, respectively.
In Figure~\ref{fig:errors-rt2q2}, we show the same results for elements
of one polynomial orser higher. The results exhibit again the expected
convergence orders.
\begin{figure}[tp]
  \centering
  \includegraphics[width=.4\textwidth]{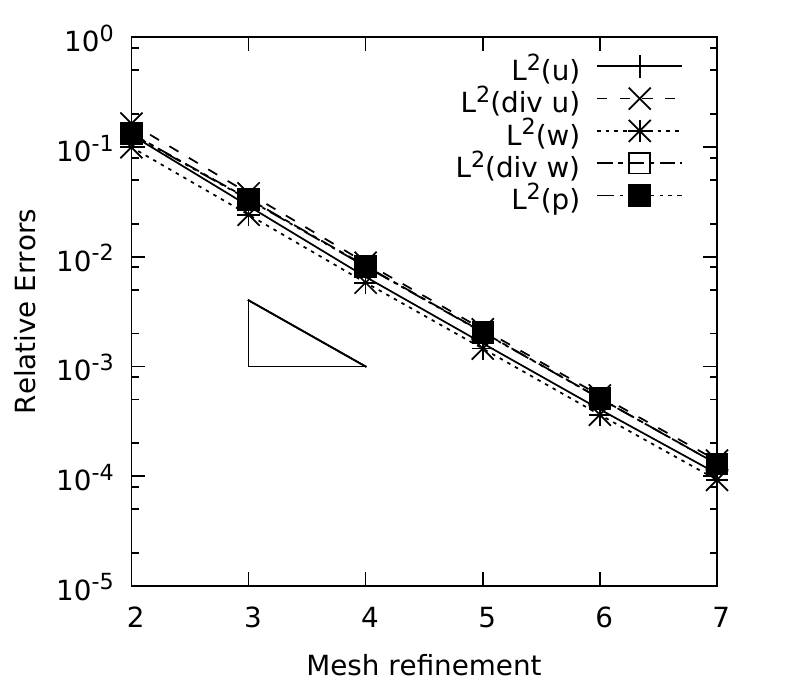}
  \includegraphics[width=.4\textwidth]{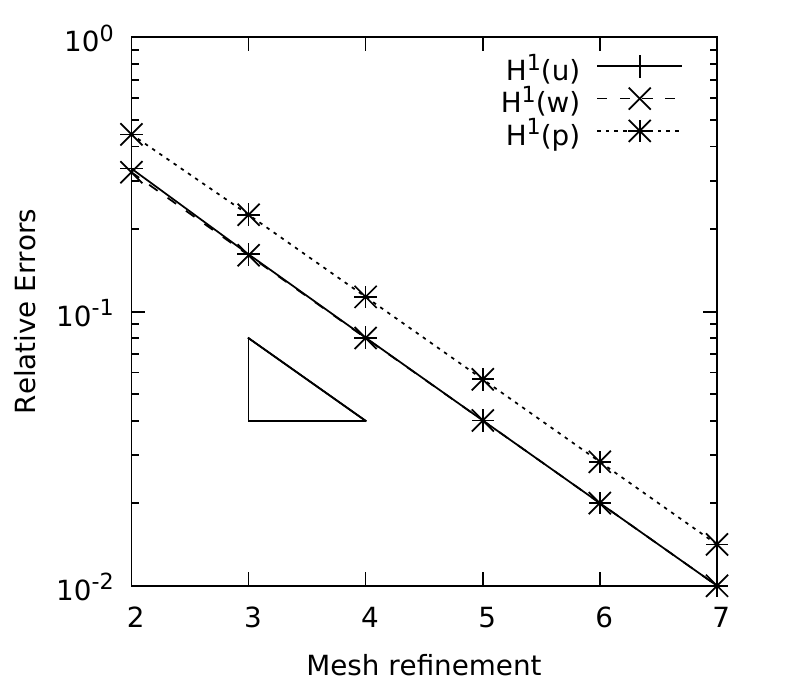}  
  \caption{Relative errors for $RT_1/Q_1$ elements. The triangle on the left indicates second order convergence, the one on the right first order.}
  \label{fig:errors-rt1q1}
\end{figure}
\begin{figure}[tp]
  \centering
  \includegraphics[width=.4\textwidth]{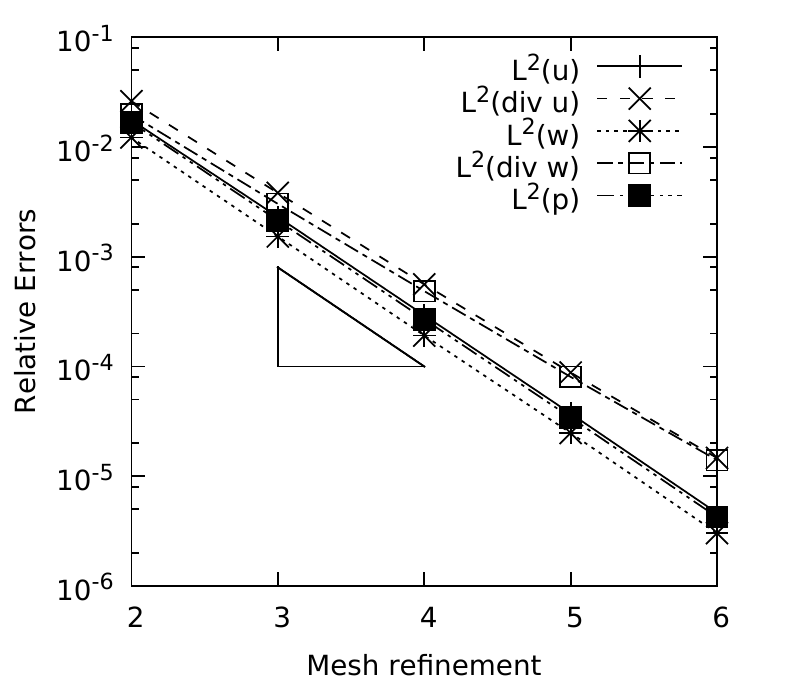}
  \includegraphics[width=.4\textwidth]{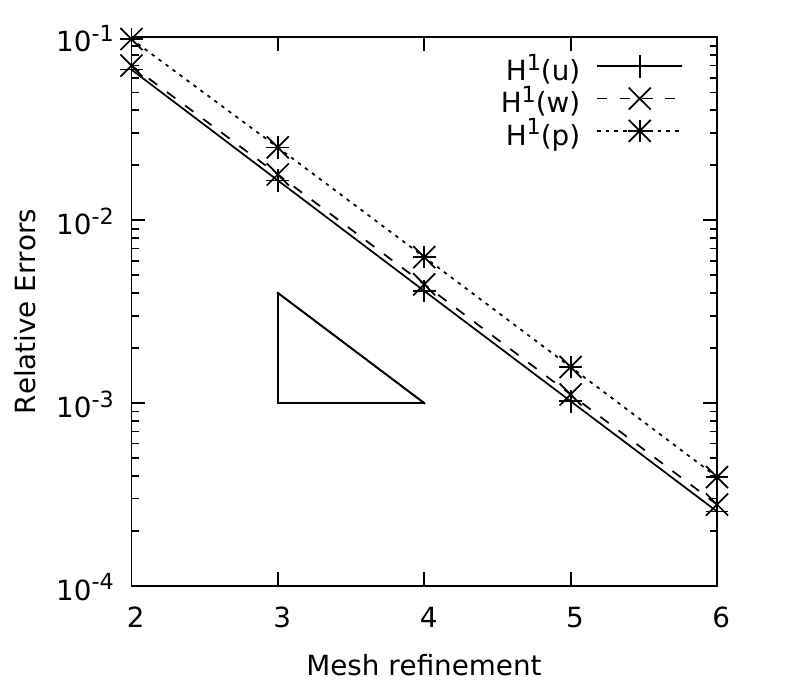}  
  \caption{Relative errors for $RT_2/Q_2$ elements. The triangle on the left indicates third order convergence, the one on the right second order.}
  \label{fig:errors-rt2q2}
\end{figure}
Details of the discretization and the time steps chosen can be found
in Table~\ref{tab:discretization}.  Due to the low accuracy of the
Euler scheme analyzed above, computations were performed with the
$\theta$-scheme, which reads for a general spatial operator $F$:
\begin{gather*}
  u_{n+1} + \theta \,\Delta t \,F(u_{n+1})  = u_n - (1-\theta) \,\Delta t\, F(u_n).
\end{gather*}
A value of $\theta=0.5$ yields the second-order Crank-Nicolson
method. We chose $\theta = 0.501$ such that the scheme is strongly
A-stable. While it is only first order, its error constant is much
smaller than for the backward Euler scheme. In any case, we chose time
steps sufficiently small such that further reduction did not improve
significant digits of the error.

\begin{table}[tp]
  \centering
  \begin{tabular}{|c|cc|cc|cc|}
    \hline
    \multicolumn{3}{|c|}{}
    &\multicolumn{2}{|c|}{$RT_1$}
    &\multicolumn{2}{|c|}{$RT_2$}
      \\\hline
    $\ell$ & $h$        & cells & $\Delta t$ & dofs  & $\Delta t$ & dofs \\\hline
    2 & \nicefrac14     & 16    & 0.08       & 352    & 0.02   & 768 \\
    3 & \nicefrac18     & 64    & 0.04       & 1344   & 0.006  & 2976\\
    4 & \nicefrac1{16}  & 256   & 0.02       & 5248   & 0.002  & 11712\\
    5 & \nicefrac1{32}  & 1024  & 0.01       & 20736  & 0.0007 & 46464\\
    6 & \nicefrac1{64}  & 4096  & 0.005      & 82432  & 0.0003 & 185088\\
    7 & \nicefrac1{128} & 16384 & 0.002      & 328704 & 0.0001 & 738816\\\hline
  \end{tabular}
  \caption{Additional data on the discretization.}
  \label{tab:discretization}
\end{table}

Finally, we verify the mass conservation of the method. Given $u_h(0)=0$,
exact mass conservation implies that at time $t>0$ there holds
\begin{gather*}
  \Delta m(t) := \storage p(t) + \biotwillis \div \disp(t)
  - \int_0^t \bigl[\div\darcy(s) - \tilde{f}_1(s) \bigr]\,ds = 0,
\end{gather*}
for the continuous in time scheme. Here, $\tilde{f}_1$ is the
$L^2$-projection of $f_1$ into the discrete pressure space $Q_h$.
Discretely in time, this identity still holds, if we replace the
integral by the quadrature rule consistent with the timestepping
scheme. In particular, the equality holds independent of approximation
quality, such that we test it on very coarse meshes and with coarse
time steps. In Table~\ref{tab:conservation}, we show results, where we
vary the parameters of the equation. In particular, $\storage \neq 0$
allows for compressible fluids and $\biotwillis\neq 1$ for some slack
in the mass balance between solid and fluid. Nevertheless, all norms
are within machine accuracy, confirming our claim.
\begin{table}
  \centering
  \begin{tabular}{|ccc|c|}
    \hline
    $\storage$ & $\biotwillis$ & $\lambda$ & $\lVert\Delta m(0.5)\rVert_{L^2(\Omega)}$
    \\\hline
    0 & 1 & 1 & 8.55e-17 \\
    0 & 0.9 & 1 & 7.36e-17\\
    0.1 & 0.9 & 1 & 7.66e-17\\
    0.1 & 0.9 & 1000 & 3.19e-14\\
    \hline
  \end{tabular}
  \caption{Verification of mass balance for various parameters. $h=\nicefrac18$,
    $\Delta t=\nicefrac1{10}$. Right hand side $f_1$ as in equation~\eqref{eq:f1-barry}.}
  \label{tab:conservation}
\end{table}

\section{Conclusions}

We presented a discretization scheme for Biot's consolidation
model which provides pointwise mass balance. It is based on a
superapproximation assumption on the divergence of the Hdiv-DG
discretization of the elasticity subproblem. The approximations of
displacement and seepage velocity, respectively, are of equal order.

\subsection*{Acknowledgements}
Computations in this article were produced using the deal.II
library~\cite{dealii84}.

\bibliography{../../../../totale,local}

\end{document}